\documentclass[12pt]{article}

\usepackage{amssymb,amsmath,amsfonts,eurosym,geometry,graphicx,color,setspace,sectsty,comment,footmisc,pdflscape,array,bbm}
\usepackage[normalem]{ulem}
\usepackage{caption}
\usepackage{subcaption}
\usepackage{hyperref}
\usepackage{tikz}
\usepackage[numbers]{natbib}
\usepackage{enumitem}

\newtheorem{theorem}{Theorem}
\newtheorem{assumption}{Assumption}

\newtheorem{corollary}[theorem]{Corollary}
\newtheorem{lemma}[theorem]{Lemma}
\newtheorem{proposition}{Proposition}
\newenvironment{proof}[1][Proof]{\noindent\textbf{#1.} }{\ \rule{0.5em}{0.5em}}

\numberwithin{equation}{section}

\newcolumntype{L}[1]{>{\raggedright\let\newline\\arraybackslash\hspace{0pt}}m{#1}}
\newcolumntype{C}[1]{>{\centering\let\newline\\arraybackslash\hspace{0pt}}m{#1}}
\newcolumntype{R}[1]{>{\raggedleft\let\newline\\arraybackslash\hspace{0pt}}m{#1}}

\begin{document}

\begin{titlepage}
\title{Positive probability of explosion for stochastic heat equation with superlinear accretive reaction term and polynomially growing multiplicative noise}

\author{
Michael Salins\thanks{Email: \href{mailto:msalins@bu.edu}{msalins@bu.edu}} \\ 
    Department of Mathematics and Statistics, Boston University
    \and
    Yuyang Zhang\thanks{Email: \href{mailto:yyz@bu.edu}{yyz@bu.edu}} \\ 
    Questrom School of Business, Boston University 
}

\date{\today}
\maketitle
\begin{abstract}
This paper studies the finite time explosion of the stochastic heat equation $\frac{\partial u}{\partial t}(t,x)=\frac{\partial^2}{\partial x^2} u(t,x)+(u(t,x))^{\beta}+\sigma(u(t,x))\dot{W}(t,x)$. We consider an interval $D=[-\pi,\pi]$ under periodic boundary condition where $\dot{W}(t,x)$ is a space-time white noise and $\sigma(u)\approx u^{\gamma}$ near $\infty$. Our results refine existing results by identifying behavior in a previously less understood regime, where we show that if $\beta\in(1,3),\gamma\in(\frac{\beta}{2},\frac{\beta+3}{4})$ or $\beta>1,\gamma\in(0,\frac{\beta}{2}]$ then mild solutions can explode with positive probability. This paper provides a partial characterization of the explosion behavior in an intermediate parameter regime, 
and contribute to the understanding of the interplay between the drift and diffusion terms.
\end{abstract}
\setcounter{page}{0}
\thispagestyle{empty}
\end{titlepage}
\pagebreak \newpage
\section{Introduction}
We explore whether solutions to a semilinear stochastic heat equation explode in finite time. The equation is
\begin{align}
    \begin{cases}
        \frac{\partial u}{\partial t}(t,x)=\frac{\partial^2}{\partial x^2} u(t,x)+b(u(t,x))+\sigma(u(t,x))\dot{W}(t,x)  &x\in D,t>0,\\
        u(t,-\pi)=u(t,\pi) & t>0,\\
        u(0,x)=u_0(x)\text{ bounded and periodic.}
    \end{cases}\label{eq:SPDE}
\end{align}
The spatial domain $D:=[-\pi,\pi]\subset\mathbb{R}$ with periodic boundary imposed. $\dot{W}(t,x)$ is a space-time white noise. We will study  functions $b$ and $\sigma$ of the form
\begin{align}
    b(u)=\begin{cases}
        u\wedge u^{\beta}, & \text{ if } 0\leq \beta<1\\
        u^{\beta}, & \text { if } \beta>1
    \end{cases} \quad\text{and}\quad \sigma(u)=\begin{cases}
        u\wedge u^{\gamma}, & \text{ if } 0\leq \gamma<1\\
        u^{\gamma}, &\text{ if } \gamma\geq 1
    \end{cases}.\label{sec1:eq:b and sigma}
\end{align}
When  $\gamma<1$ or $\beta <1$, we take the minimum $u\wedge u^\beta$ or $u \wedge u^\gamma$  to ensure local Lipschitz continuity. This modification preserves the asymptotic behavior at infinity, as $\sigma(u)$ exhibits the same growth rates as the corresponding power functions for large $u$. Throughout this paper, we focus on positive solutions. Assuming the condition $u_0\geq 0$, this ensures that the solution remains positive due to the comparison principle \cite{Kotelenez1992,Mueller199101}.

If there doesn't exist a noise term then the semilinear stochastic heat equation reduces to a semilinear partial differential equation. Any nontrivial nonnegative solution of this PDE will blow up in finite time if $\beta>1$ follows from a classical Kaplan-type argument \cite{Kaplan}. Intuitively, if a noise term is added to the PDE as in \eqref{eq:SPDE}, there still exists a region of reaction-induced explosion where $\gamma$ is small enough so that the reaction term $(u(t,x))^{\beta}$ is dominant. We will prove that in the case that $\beta\in(1,3),\gamma\in(\frac{\beta}{2},\frac{\beta+3}{4})$ or $\beta>1,\gamma\in(0,\frac{\beta}{2}]$ solutions can explode with positive probability.  
Combining with the results in the literature, we have the following conclusion for the parameters in different regions, which are drawn in Figure~\ref{fig:1}, below:

\begin{center}
    \includegraphics[width=\textwidth]{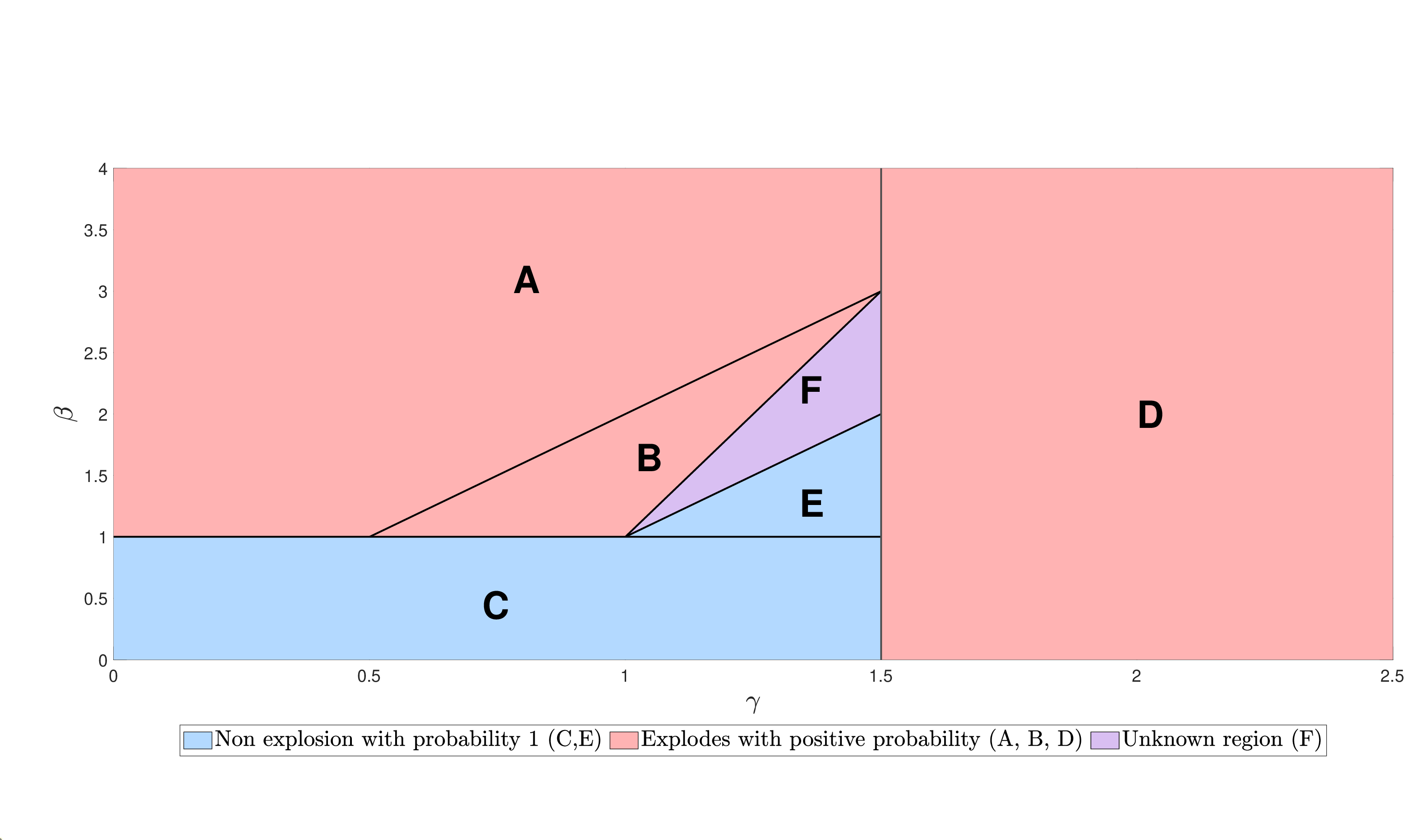}
    \captionof{figure}{Explosion regions}
    \label{fig:1}
\end{center}

\begin{itemize}[labelwidth=2em, labelsep=0.5em, align=left]
    \setlength{\itemsep}{0pt}
    \setlength{\parskip}{0pt}
    \setlength{\parsep}{0pt}
    \item [(A)]If ${\beta>1}$ and $\gamma\in (0,\frac{\beta}{2}]$, we prove that solutions can explode with positive probability in Section \ref{S:holder-case}.
    \item [(B)]If $\beta\in(1,3)$ and $\gamma\in(\frac{\beta}{2},\frac{\beta+3}{4})$, we prove that solutions can explode with positive probability in Section \ref{S:non-holder-case}.
    \item [(C)] If $\beta\in [0,1]$ and $\gamma\in [0, \frac{3}{2}]$ then the solutions cannot explode by \cite{Mueller1991,salins2025} and comparison principle \cite{Kotelenez1992,Mueller199101}.
    \item [(D)]If ${\beta \geq 0}$ and $\gamma > \frac{3}{2}$, then results of Mueller \cite{Mueller2000}, together with a comparison principle \cite{Kotelenez1992,Mueller199101}, imply that the solutions explode with positive probability.
    \item [(E)]If $\beta\in (1,2)$ and $\gamma\in (\frac{1+\beta}{2},\frac{3}{2}]$, then the solutions cannot explode by \cite{Mickey202409}.
    \item [(F)]If ${\beta  \in (1,3)}$ and $\gamma\in [\frac{\beta+3}{4},\frac{1+\beta}{2}{\wedge \frac{3}{2}}]$ then the problem of explosion remains open but we conjecture the solutions cannot explode.
\end{itemize}
Also, explosion on some boundaries is known, for example, if $\beta >1$ and  $\gamma=0$ then solutions explode almost surely \cite{Bonder2009,Foondun2021} and if $\beta=1$ and $\gamma=\frac{1+\beta}{2}$ then for $b(u)=Au^{\beta}$ solutions cannot explode provided that $A$ is sufficiently small \cite{Mickey202409}.

In a series of papers \cite{Mueller1991,Mueller1998,Mueller2000,Mueller1993}, 
Mueller and Sowers identified a critical growth rate for $\sigma(u)$ in the absence of a drift term ($b(u)\equiv 0$), under periodic boundary conditions, driven by space-time white noise $\dot{W}$ in one spatial dimension. Specifically, if $|\sigma(u)| \leq C(1+|u|^{\gamma})$ for some $\gamma < \frac{3}{2}$, then solutions do not exhibit finite-time blow-up almost surely. Conversely, if $\sigma(u) = c|u|^{\gamma}$ for some $c>0$ and $\gamma > \frac{3}{2}$, then the solutions blow up with positive probability. In the critical case $\gamma = \frac{3}{2}$, solutions remain non-explosive \cite{salins2025}.

A substantial body of work has extended these results to a range of settings, 
including higher-dimensional domains with Dirichlet boundary conditions under colored noise 
\cite{Franzova1999,Yuyang2025}, 
reaction--diffusion equations on unbounded domains \cite{Krylov1996}, 
fractional heat equations \cite{Bezdek2018,Foondun2019}, 
as well as nonlinear Schr\"odinger and stochastic wave equations 
\cite{Debuss2002,Mueller1997}. 
More recently, attention has focused on the role of superlinear drift terms $b(u(t,x))$ 
in influencing finite-time explosion for equations of the form \eqref{eq:SPDE}.

Building on earlier work of Bonder and Groisman \cite{Bonder2009}, 
Foondun and Nualart \cite{Foondun2021} established that when $\sigma$ is uniformly bounded above and below, 
that is, $0 < c \leq \sigma(u) \leq C < \infty$, and $b$ is positive and increasing, 
the explosion behavior is completely determined by Osgood condition. 
Specifically, almost sure blow-up occurs if
\begin{align}
    \int_1^\infty \frac{1}{b(u)}\,du < \infty, \label{Osgood explosive condition}
\end{align}
whereas global existence holds almost surely when
\begin{align}
    \int_1^\infty \frac{1}{b(u)}\,du = \infty.
\end{align}
This criterion coincides with Osgood’s classical condition for blow-up in the ordinary differential equation $\frac{dv}{dt} = b(v(t))$ \cite{Osgood1898}. 
Moreover, on unbounded domains, the condition \eqref{Osgood explosive condition} leads to instantaneous blow-up everywhere \cite{Foondun2024}.

Despite these advances, the combined effect of superlinear drift and superlinear diffusion remains poorly understood. 
An initial step in this direction was taken by Dalang, Khoshnevisan, and Zhang \cite{Dalang2019}, 
who proved global existence under the conditions $b(u) \leq C(1+|u|\log|u|)$ and 
$\sigma(u) \in o(|u|(\log|u|)^{1/4})$. 
Subsequent work has extended these results to broader classes of drift terms satisfying generalized Osgood conditions, 
as well as to more general domains, noise structures, and related equations 
\cite{AGRESTI2023,chen.huang:23:superlinear,Liang2022,Annie2021,salins2024,salins2025}.

More recently, Salins \cite{Mickey202409} showed that Mueller’s critical growth rate 
$\sigma(u) \approx |u|^{3/2}$ remains sharp for ensuring global solutions, 
even when the drift is accretive and superlinear. 
In this regime, strong stochastic fluctuations can counterbalance the growth induced by $b$, 
keeping the $L^1$ norm finite and effectively reducing the dynamics to a setting similar to that studied by Mueller 
\cite{Mueller1991,Mueller1998,Mueller2000,Mueller1993}. 
In particular, for $b(u)=A u^\beta$ and $\sigma(u)=u^\gamma$, 
non-explosion holds when $\beta \in (1,2)$ and $\gamma \in \bigl(\frac{1+\beta}{2}, \frac{3}{2}\bigr]$, 
and also in the boundary case $\gamma=\frac{1+\beta}{2}$ provided that $A$ is sufficiently small.

Joseph and Ovhal \cite{joseph2026} proved explosion with positive probability if there exists $\eta>0$ such that $\frac{b(x)}{x}\leq (\frac{\sigma(x)}{x})^{\frac{1}{4}-\eta}$ for $x\geq 1$ and $b(u), \sigma(u) \leq u(\log(u))^\theta$ for some $\theta>0$. Even though their assumptions do not allow $b(u) = u^\beta$ and $\sigma(u) = u^\gamma$ for $\beta,\gamma>1$, we can show that their condition $\frac{b(x)}{x}\leq (\frac{\sigma(x)}{x})^{\frac{1}{4}-\eta}$ is still valid in our setting. When $b(x) = x^\beta$ and $\sigma(x) = x^\gamma$ this condition is equivalent to $(\gamma -1) < \frac{(\beta -1)}{4}$, or equivalently $\gamma < \frac{\beta + 3}{4}$.

Inspired by \cite{joseph2026}, the results of this paper prove that if $\beta\in(1,3),\gamma\in(\frac{\beta}{2},\frac{\beta+3}{4})$ or $\beta>1,\gamma\in(0,\frac{\beta}{2}]$ then there exists an initial condition $u_0\in C(\overline{D})$ such that solutions of \eqref{eq:SPDE} explode with positive probability under Assumption \ref{assumption1} and \ref{assumption2}. Furthermore, we can show that there is a positive probability of explosion from any initial data satisfying Assumption \ref{assumption2} and $u(0,x)\not\equiv 0$ a.e. by using the support theorem of \cite{Bally1995}. After combining the results of this paper with \cite{Mickey202409}, there still remains a regime where the question of finite-time explosion is unsolved: $\beta \in (1,3),\gamma\in [\frac{\beta+3}{4} ,\frac{\beta+1}{2} \wedge \frac{3}{2})$.  We conjecture that in this regime, solutions cannot explode, but we leave the proof of non-explosion for future work. Our intuition is that in this regime the noisy term is dominant and we expect the behavior to match the non-explosive behavior of \cite{Mueller1998,Mueller1993,salins2025}.
 
We prove the results with the help of $L^1$ norm of the local mild solution. Because of the periodic boundary condition the $L^1$ norm $I(t):=\int u(t,x)dx$ of the local mild solution is a local submartingale. Because when $\beta>1$ we have $b(u)=u^{\beta}$ and by using the stopping time $\tau_n^\infty$ as defined in \eqref{def:tau_n} it follows that
\begin{align}
    I_n(t)&:= I(t \wedge \tau^\infty_n)\nonumber\\
    &=I_n(0)+\int_0^{t\land\tau_n^\infty}\int (u(t,x))^{\beta}dx+\int_0^{t\land\tau_n^\infty}\int\sigma(u(t,x))W(dxds).  
\end{align}

In this case, we can directly apply It\^{o} formula to $V(I_n(t)):=1-(I_n(t))^{-\epsilon}$ for any $\epsilon>0$ then we get
\begin{align}
    V(I_n(t))=&V(I_n(0))\nonumber\\
    &+\int_0^{t\land\tau_n^\infty}\int \left( \frac{\epsilon (u(t,x))^{\beta} }{(I_n(s))^{\epsilon+1}}-\frac{\epsilon(\epsilon+1)(\sigma(u(t,x)))^2}{2(I_n(s))^{\epsilon+2}}\right)dxds+N(t\land \tau_n^\infty),
\end{align}
where $N(t\land\tau_n^\infty)$ is a local martingale after introducing suitable initial condition and stopping times. This It\^{o} formula depends on  both the spatial $L^{\beta}$ and $L^{2\gamma}$ norms of the local mild solution $u(t,x)$ by \eqref{sec1:eq:b and sigma}.

In the case $\beta>1$ and $2\gamma\leq \beta$, by H\"older inequality,
\begin{align} 
   \int b(u(t,x)) dx &=\int (u(t,x))^\beta dx \nonumber\\
   &\geq C \left(\int (u(t,x))^{2\gamma}dx\right)^{\frac{\beta}{2\gamma}}\nonumber\\
   &\geq C \left(\int (\sigma(u(t,x)))^2dx\right)^{\frac{\beta}{2\gamma}}.\label{eq:holder}
\end{align}
Then for suitable initial condition and stopping times, it follows that $V(I_n(t))$ is a bounded local submartingale. Using the Khasminskii's method as in stochastic ordinary differential equations (see \cite[Theorems 3.5 and 3.6]{Khasminskii}) we can prove $L^1$ norm $I(t)$ or $L^\infty$ norm of $u(t,x)$ explodes with positive probability in finite time. Furthermore, because $I(t)\leq C|u|_{L^\infty}$ we can prove that solutions explodes with positive probability. We prove positive probability of explosion in this setting in Section \ref{S:holder-case}, below.

In the case $\beta\in (1,3)$ and $\gamma\in(\frac{\beta}{2},\frac{\beta+3}{4})$, the H\"older inequality points in the reverse direction of \eqref{eq:holder}, and cannot help us prove explosion. Instead of using H\"older's inequality, we prove that there is a positive probability that the $L^{2\gamma}$ norm stays sufficiently small compared to the $L^\beta$ norms so that $V$ remains a bounded supermartingale. This requires us to develop moment bounds on the Lebesgue and stochastic convolution integrals in the integral equation of the local mild solution $u(t,x)$
\begin{align}
	u(t,x) = &\int G(t,x-y) u(0,y)dy + \int_0^t \int G(t-s,x-y)(u(s,y))^\beta dyds \nonumber\\
	&+ \int_0^t \int G(t-s,x-y)\sigma(u(s,y)) W(dyds). \label{eq:integral equation intro}
\end{align}
Using these bounds with suitable initial condition and stopping time, we can prove that with positive probability the $L^1$ norm and $L^\infty$ norm stay relatively close and the $L^1$ norm reaches $\infty$ in a finite amount of time. Thus, we can prove that solutions explodes with positive probability. We prove non-explosion in this setting in Section \ref{S:non-holder-case}, below.

In Section \ref{S:setup} we will present assumptions, definitions and the main result for this paper. In Section \ref{S:Ito of L1} we will show the It\^{o} formula of $L^1$ norm of the local mild solution $u(t,x)$ and the corresponding Lebesgue integral and quadratic variation bound. In Section \ref{S:holder-case} we will prove positive probability explosion in the case $\beta> 1$ and $\gamma\in(0,\frac{\beta}{2}]$ where H\"older inequality \eqref{eq:holder} is helpful. Finally in Section \ref{S:non-holder-case} we will prove positive probability explosion in the case $\beta\in (1,3)$ and $\gamma\in (\frac{\beta}{2},\frac{\beta+3}{4})$ by developing and utilizing moment bounds for Lebesgue and stochastic convolution integrals.

\section{Assumptions, definitions and results}\label{S:setup}
Let $D:=[-\pi,\pi]$ with periodic boundary by the spatial domain. Define the periodic heat kernel for $x\in [-\pi,\pi]$ by
\begin{align}
    G(t,x)=\frac{1}{2\pi}+\frac{1}{\pi}\sum_{k=1}^\infty e^{-|k|^2t}\cos(kx).\label{eq:heat kernel}
\end{align}
\begin{assumption}\label{assumption1}
In \eqref{eq:SPDE}, the functions $b:\mathbb{R}\mapsto\mathbb{R}$ and $\sigma:\mathbb{R}\mapsto\mathbb{R}$ are specified as
\begin{align}
    b(u)=u^{\beta}\quad\text{and}\quad \sigma(u)=\begin{cases}
        u^{\gamma},\gamma\geq 1\\
        u\wedge u^{\gamma},\gamma<1
    \end{cases}, \label{sec2:eq:b and sigma}
\end{align}
we assume $\beta$ and $\gamma$ are real numbers and satisfy
\begin{align}
    \beta\in (1,3),\gamma\in \left(\frac{\beta}{2},\frac{\beta+3}{4}\right)\text{ or } \beta>1,\gamma\in (0,\frac{\beta}{2}].\label{def:b and sigma}
\end{align}
\end{assumption}
For any $n\in\mathbb{N}$, define the cutoff versions of $b$ and $\sigma$ by
\begin{align}
    b_n(u)=\begin{cases}
        b(-n),\text{ if }u<-n\\
        b(u),\text{ if }u\in(-n,n)\\
        b(n),\text{ if }u>n.
    \end{cases}\label{eq:cutoff of b}\\
    \sigma_n(u)=\begin{cases}
        \sigma(-n),\text{ if }u<-n\\
        \sigma(u),\text{ if }u\in(-n,n)\\
        \sigma(n),\text{ if }u>n.
    \end{cases}\label{eq:cutoff of sigma}
\end{align}
Because $b$ and $\sigma$ satisfy \eqref{sec2:eq:b and sigma} and \eqref{def:b and sigma}, the cutoff functions $b_n$ and $\sigma_n$ are globally Lipschitz continuous. The mild solution to the cutoff SPDE
\begin{align}
     \frac{\partial u_n}{\partial t}(t,x)=\frac{\partial^2}{\partial x^2} u_n(t,x)+b_n(u_n(t,x))+\sigma_n(u_n(t,x))\dot{W}(t,x)\label{eq:SPDE of cutoff}
\end{align}
is defined to be the solution to the integral equation
\begin{align}
    \begin{split}
        u_n(t,x)=&\int G(t,x-y)u(0,y)dy+\int_0^t\int G(t-s,x-y)b_n(u_n(s,y))dyds\\
        &+\int_0^t\int G(t-s,x-y)\sigma_n(u_n(s,y))W(dyds).
    \end{split}\label{eq:integral equation of cutoff}
\end{align}
Because $b_n$ and $\sigma_n$ are globally Lipschitz continuous, classical results prove that for each $n$, there exists a unique mild solution $u_n$ \cite{DaPrato_Zabczyk_2014,Dalang1999}. Furthermore, these solutions are all \textit{consistent} in the sense that for any $n<m$
\begin{align}
    u_n(t,x)=u_m(t,x)\text{ for all }x\in D\text{ and }t\in[0,\tau_n^\infty]\label{eq:consistency}
\end{align}
where
\begin{align}
    \tau_n^\infty:=\inf\Big\{t>0:\sup_{x\in D}|u_n(t,x)|>n\Big\}.\label{def:tau_n}
\end{align}
The consistency is a consequence of the uniqueness of each $u_n(t,x)$ and the fact that $b_n(u)=b_m(u)$ and $\sigma_n(u)=\sigma_m(u)$ for all $|u|\leq n$.

Define the \textit{explosion time} by
\begin{align}
    \tau_\infty^\infty=\sup_n\tau_n^\infty.\label{def:tau_infty}
\end{align}
We can uniquely define a \textit{local mild solution }by
\begin{align}
    u(t,x):=u_n(t,x)\text{ for all }x\in D,t\in[0,\tau_n^\infty].\label{def:local mild solution}
\end{align}
If $\tau_\infty^\infty<+\infty$, then we say that $u(t,x)$ explodes in finite time. The local mild solution is called a global mild solution if
\begin{align}
    \mathbb{P}(\tau_\infty^\infty=\infty)=1.
\end{align}
We further assume that the initial condition $u_0\geq 0$ as shown in Assumption \ref{assumption2}.
\begin{assumption}\label{assumption2}
    \begin{align}
        u(0,x)\geq 0 \text{ for a.e. } x\in D. \label{assum: positivity}
    \end{align}
\end{assumption}
Assumption \ref{assumption2} combined with $\sigma(0)=0$ guarantees that the local mild solution remains positive due to the comparison principle \cite{Kotelenez1992,Mueller199101}.
\begin{theorem}\label{thm:main result}
    Assume Assumption \ref{assumption1} and \ref{assumption2}, there exists initial condition $u_0\in C(\overline{D})$ such that if $u(0,x)=u_0(x)$ then $u(t,x)$ explodes in finite time with strictly positive probability, i.e. $\mathbb{P}(\tau_\infty^\infty<\infty)>0$.
\end{theorem}
Using the support theorem \cite[Theorem 2.1]{Bally1995} we can show that there is a positive probability of explosion from any initial data satisfying Assumption \ref{assumption2} and $u(0,x)\not\equiv 0$ a.e., not just large initial data.

Throughout the paper, the constant $C>0$ can refer to any positive constant and its value can change from line to line. An integral symbol without explicit bounds refers to integration over the spatial domain
\begin{align*}
    \int v(x)dx:=\int_Dv(x)dx=\int_{-\pi}^\pi v(x)dx.
\end{align*}
We use $L^p$ to refer to the standard $L^p$ spaces over the spatial domain with the norms
\begin{align}
    |v|_{L^p}:=\Big(\int|v(x)|^pdx\Big)^{\frac{1}{p}}\text{ and }|v|_{L^\infty}:=\sup_{x\in[-\pi,\pi]}|v(x)|.\label{def:L^p norm}
\end{align}

\section{\texorpdfstring{It\^{o} formula of the $L^1$ norm of $u(t,x)$}{Ito formula of L1 norm of u(t,x)}}\label{S:Ito of L1}
\begin{proposition}\label{prop:positivity}
Assume Assumption \ref{assumption1}-\ref{assumption2}, let $u(t,x)$ be the local mild solution to \eqref{eq:SPDE} then
    \begin{align}
        u(t,x)\geq 0 \text{ for all }t\geq 0\text{ and a.e. }x\in D.\label{eq:positivity}
    \end{align}
\end{proposition}
This is the result of the comparison principle \cite{Kotelenez1992,Mueller199101} as mentioned before. Let $\tau_n^\infty$ be defined in \eqref{def:tau_n} and define
\begin{align}
    I_n(t):=\int u(t\land\tau_n^\infty,x)dx.\label{def: I_n}
\end{align}
Because the spatial integral of the periodic heat kernel $\int G(t,x-y)\equiv 1$ for all $t>0$ and $y\in [-\pi,\pi]$, $I_n(t)$ is a nonnegative semimartingale solving
\begin{align}
    \begin{split}
        I_n(t)=&I_n(0)+\int_0^{t\land\tau_n^\infty}\int b(u(s,x))dxds+\int_0^{t\land\tau_n^\infty}\int\sigma(u(s,x))W(dxds).
    \end{split} \label{eq:martingale rep of I_n}
\end{align}
Define $I(t)$ such that 
\begin{align}
    I(t):=I_n(t),\textrm{ for }t\in [0,\tau_n^\infty].\label{def:I_t}
\end{align}
It follows that $I(t)$ is a nonnegative local submartingale for $t<\tau_\infty^\infty$. Define the $L^1$ stopping times for $M>0$
\begin{align}
    \tau_M^1:=\inf\{t\in[0,\tau_\infty^\infty):|u(t)|_{L^1}>M\}\label{def:tau_M_1}
\end{align}
 We have the following bound for Lebesgue integral in the semimartingale representation of $I(t)$
\begin{lemma}\label{lem:Lebesgue int}
    For any $M>0$, the Lebesgue integral in the semimartingale representation of $|u(t)|_{L^1}$ satisfies
    \begin{align}
        \mathbb{E}\int_0^{\tau_M^1\land\tau_\infty^\infty}\int b(u(s,y))dyds\leq M.\label{ineq:Lebesgue int}
    \end{align}
\end{lemma}
\begin{proof}
According to \eqref{eq:martingale rep of I_n}, for any $M>0,n>0$, by martingale property
\begin{align}
    \mathbb{E}(I_n(t\land\tau_M^1))=I_n(0)+\mathbb{E}\int_0^{t\land\tau_n^\infty\land\tau_M^1}\int b(u(s,y))dyds.
\end{align}
Because each term of the right-hand side is nonnegative by Assumption \ref{assumption1} and \ref{assumption2}, it follows that
\begin{align}
     \mathbb{E}\int_0^{t\land\tau_n^\infty\land \tau_M^1}\int b(u(s,y))dyds\leq M.
\end{align}
This bound doesn't depend on $n$ or $t$. Thus, \eqref{ineq:Lebesgue int} holds.
\end{proof}

As in \cite{salins2025} we have the following bound for quadratic variation.
\begin{lemma}\label{lem:quad var}
    For any $M>0$, the quadratic variation of $|u(t)|_{L^1}$ satisfies
    \begin{align}
        \mathbb{E}\int_0^{\tau_M^1\land\tau_\infty^\infty}\int |\sigma(u(t,x))|^2dyds\leq M^2.\label{ineq:quad var}
    \end{align}
\end{lemma}
\begin{proof}
   Applying It\^{o} formula to \eqref{eq:martingale rep of I_n}, for any $M>0,n>0$,
   \begin{align}
       &\mathbb{E}(I_n(t\land\tau_M^1))^2 \nonumber\\
       =&\mathbb{E}(I_n(0))^2+2\mathbb{E}\int_0^{t\land\tau_n^\infty\land\tau_M^1}\int b(u(s,y))I_n(s)dyds\nonumber\\
       &+\mathbb{E}\int_0^{t\land\tau_n^\infty\land \tau_M^1}\int |\sigma(u(s,y))|^2dyds.
   \end{align}
   Each term on the right-hand side is nonnegative by Assumption \ref{assumption1} and \ref{assumption2}, and $\mathbb{E}(I_n(t\land\tau_M^1))^2\leq M^2$ by the definition of $\tau_M^1$. Therefore,
   \begin{align}
       \mathbb{E}\int_0^{t\land\tau_n^\infty\land \tau_M^1}\int |\sigma(u(s,y))|^2dyds\leq M^2.
   \end{align}
   This bound doesn't depend on $n$ or $t$. Thus, \eqref{ineq:quad var} holds.
\end{proof}

For any $\epsilon>0$, let $V(x):=1-x^{-\epsilon}$ then by It\^{o} formula and $b(u)=u^{\beta}$ it follows that
\begin{align}
     V(I_n(t))=V(I_n(0))+\int_0^{t\land\tau_n^\infty}\int\left( \frac{\epsilon (u(s,x))^{\beta}}{(I_n(s))^{\epsilon+1}}-\frac{\epsilon(\epsilon+1)(\sigma(u(s,x)))^2}{2(I_n(s))^{\epsilon+2}}\right)dxds+N(t\land \tau_n^\infty), \label{eq:Ito of V}
\end{align}
where $N(t\land\tau_n^\infty)$ is given by
\begin{align}
    N(t\land\tau_n^\infty):=\int_0^{t\land\tau_n^\infty}\int\frac{\epsilon\sigma(u(s,x))}{(I_n(s))^{\epsilon+1}}W(dxds).
\end{align}
We next prove Theorem \ref{thm:main result} case by case.

\textbf{Remark}: In the proof of Theorem \ref{thm:main result} we focus on the case that $I_n(s)$ always stays away from $0$ by introducing suitable initial condition and stopping times. Thus, the It\^{o} formula can be directly applied and $N(t\land\tau_n^\infty)$ is a local martingale. The local martingale will be true martingale after introducing the suitable stopping times.

\section{\texorpdfstring{Proof of Theorem \ref{thm:main result}: Case $\beta>1, \gamma\in (0,\frac{\beta}{2}]$}{Proof of Theorem 1: Case gamma in (0,beta/2)}} \label{S:holder-case}
We will prove Theorem \ref{thm:main result} in the region A of Figure \ref{fig:2}
\begin{figure}[!htbp]
    \centering
    \includegraphics[width=\textwidth]{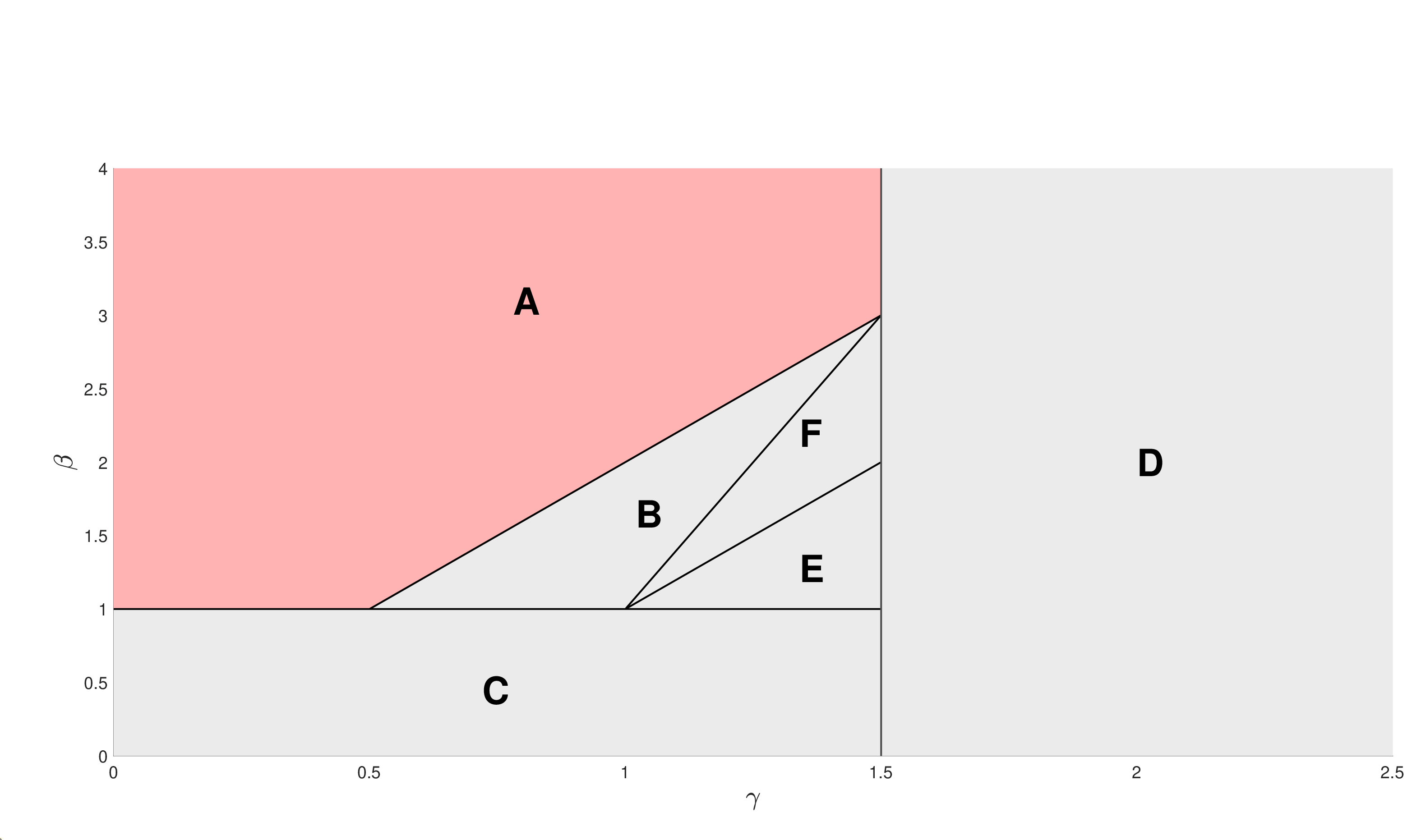}
    \caption{Explosion regions Section \ref{S:holder-case}}
    \label{fig:2}
\end{figure}

The case that $\beta>1,\gamma\in (0,\frac{\beta}{2}]$ where H\"older inequality is helpful is region A and part of region D of Figure \ref{fig:2}. What we need to prove positive probability for explosion is region A and the other part for this region where H\"older inequality is helpful is the results by Mueller \cite{Mueller2000} together with a comparison principle \cite{Kotelenez1992,Mueller199101}. Recall It\^{o} formula for $V(I_n(t)):=1-(I_n(t))^{-\epsilon}$ is given by
\begin{align}
     V(I_n(t))=V(I_n(0))+\int_0^{t\land\tau_n^\infty}\int \left(\frac{\epsilon (u(s,x))^{\beta}}{(I(s))^{\epsilon+1}}-\frac{\epsilon(\epsilon+1)(\sigma(u(t,x)))^2}{2(I(s))^{\epsilon+2}}\right)dxds+N(t\land \tau_n^\infty).\label{eq:ito of V case 1}
\end{align}
For simplicity we drop the subscription of $I_n(t)$ in the drift term of \eqref{eq:ito of V case 1}. It follows that 
by H\"older inequality \eqref{eq:holder}, because $2\gamma\leq\beta$ 
\begin{align}
    \int(\sigma(u(t,x)))^2dx\leq \int(u(s,x))^{2\gamma}dx\leq (2\pi)^{\frac{\beta-2\gamma}{\beta}}\Big(\int(u(s,x))^{\beta}dx\Big)^{\frac{2\gamma}{\beta}}.\label{ineq:holder of gamma}
\end{align}
Also, because $\beta>1$ it also follows by H\"older inequality
\begin{align}
   (2\pi)^{\beta-1} \int(u(s,x))^{\beta}dx\geq  (I(s))^{\beta}.\label{ineq:holder of beta}
\end{align}
Thus, using \eqref{ineq:holder of gamma} we have the following inequality for the drift component of \eqref{eq:ito of V case 1}
\begin{align}
   \begin{split}
       A:= &\int\left( \frac{\epsilon (u(s,x))^{\beta}}{(I(s))^{\epsilon+1}}-\frac{\epsilon(\epsilon+1)(\sigma(u(s,x)))^2}{2(I(s))^{\epsilon+2}}\right)dx\\
    \geq & \int \frac{\epsilon (u(s,x))^{\beta}}{(I(s))^{\epsilon+1}}dx-(2\pi)^{\frac{\beta-2\gamma}{\beta}}\frac{\epsilon(\epsilon+1)}{2(I(s))^{\epsilon+2}}\Big(\int(u(s,x))^{\beta}dx\Big)^{\frac{2\gamma}{\beta}}.
   \end{split}\label{ineq:A step 1}
\end{align}
Combining \eqref{ineq:A step 1} with \eqref{ineq:holder of beta} then we get
\begin{align}
    A\geq (2\pi)^{1-\beta}\epsilon\Big(1-\frac{(2\pi)^{\beta-2\gamma}(\epsilon+1)}{2(I(s))^{\beta+1-2\gamma}}\Big)(I(s))^{\beta-\epsilon-1}\label{ineq:A step 2}.
\end{align}
Because $\beta>1$, we can choose $\epsilon=\beta-1$. For any $\delta \in (0, (2\pi)^{1-\beta}\varepsilon)$ there exists $C_{\beta,\gamma,\delta}>0$ such that if $I(s)\geq C_{\beta,\gamma,\delta}$ then
\begin{align}
    A\geq \delta.\label{ineq: A}
\end{align}
Choose such a $\delta$ and assume $I(0)>C_{\beta,\gamma,\delta}$ then for $n\in\mathbb{N}$ such that $n>I(0)$, we define the stopping time for $L^1$ norm of local mild solution 
\begin{align}
    \tilde{\tau}_n^1:=\inf\{t\in [0,\tau_\infty^\infty):I(t)\leq C_{\beta,\gamma,\delta} \text{ or }I(t)\geq n\}.\label{def:tilde_tau_n}
\end{align}
Then for any $T>0$, by it follows that by \eqref{eq:ito of V case 1}, \eqref{ineq: A} and definition of $\tilde{\tau}_n^1$
\begin{align}
   \begin{split}
        V(I(\tilde{\tau}_n^1\land T))=& V(I(0))+\int_0^{\tilde{\tau}_n^1\land T\land\tau_\infty^\infty}\int \frac{\epsilon (u(s,x))^{\beta}}{(I(s))^{\epsilon+1}}-\frac{\epsilon(\epsilon+1)(\sigma(u(s,x)))^2}{2(I(s))^{\epsilon+2}}dxds\\
    &+N(\tilde{\tau}_n^1\land T\land \tau_\infty^\infty)\\
    \geq & V(I(0))+ \delta(\tilde{\tau}_n^1\land T\land\tau_\infty^\infty)+N(\tilde{\tau}_n^1\land T\land \tau_\infty^\infty).
   \end{split}\label{ineq: submartingale for L1}
\end{align}
Because $\delta(\tilde{\tau}_n^1\land T\land\tau_\infty^\infty)\geq 0$, we can take the expectation for both sides of \eqref{ineq: submartingale for L1} to get
\begin{align}
  \mathbb{E}V(I(\tilde{\tau}_n^1\land T))\geq V(I(0)).\label{ineq:supermart property}  
\end{align}
There are 4 possibilities for $I(\tilde{\tau}_n^1\land T)$. Thus,
\begin{align}
    \mathbb{E}V(I(\tilde{\tau}_n^1\land T))=& \mathbb{E}\Big( V(I(\tilde{\tau}_n^1\land T))\mathbbm{1}_{\{I(\tilde{\tau}_n^1\land T)=C_{\beta,\gamma,\delta}\}}\Big)+\mathbb{E}\Big( V(I(\tilde{\tau}_n^1\land T))\mathbbm{1}_{\{I(\tilde{\tau}_n^1\land T)=n\}}\Big)\nonumber\\
    &+ \mathbb{E}\Big( V(I(\tilde{\tau}_n^1\land T))\mathbbm{1}_{\{\tau_\infty^\infty<T\land\tilde{\tau}_n^1\}}\Big)+\mathbb{E}\Big( V(I(\tilde{\tau}_n^1\land T))\mathbbm{1}_{\{T<\tau_\infty^\infty\land\tilde{\tau}_n^1\}}\Big)
\end{align}
Because $V$ is positive and bounded by one, it follows that
\begin{align}
    \mathbb{E}V(I(\tilde{\tau}^1_n \wedge T)) \leq & V(C_{\beta,\gamma,\delta})+  \mathbb{P}(\tilde{\tau}^1_n\land\tau_\infty^\infty >T) \nonumber\\
    &+ \mathbb{P}\Big(\{I(\tilde{\tau}^1_n) = n,\tilde{\tau}_n^1<T\land\tau_\infty^\infty\} \text{ or } \tau^\infty_\infty <T\Big).\label{ineq:explosion prob step 1}
\end{align}
Thus,
\begin{align}
   &\mathbb{P}\Big(\{I(\tilde{\tau}^1_n) = n,\tilde{\tau}_n^1<T\land\tau_\infty^\infty\} \text{ or } \tau^\infty_\infty <T\Big)\nonumber\\
   \geq & V(I(0)) - V(C_{\beta,\gamma,\delta}) - \mathbb{P}(\tilde{\tau}^1_n\land\tau_\infty^\infty > T)\label{ineq:explosion prob step 2}
\end{align}
By \eqref{ineq: submartingale for L1}, positivity of $V$ and boundedness of $V$ it follows that
\begin{align}
    \delta\mathbb{E}(\tilde{\tau}_n^1\land T\land\tau_\infty^\infty)\leq 1.
\end{align}
Thus, by Chebyshev inequality
\begin{align}
    \mathbb{P}(\tilde{\tau}_n^1\land\tau_\infty^\infty>T)\leq \frac{\mathbb{E}(\tilde{\tau}_n^1\land T\land\tau_\infty^\infty)}{T}\leq \frac{1}{\delta T}.\label{ineq:chebyshev case 1}
\end{align}
Thus, it follows from \eqref{ineq:explosion prob step 2} and \eqref{ineq:chebyshev case 1} that for any $T>0$
\begin{align}
        &\mathbb{P}\Big(\{I(\tilde{\tau}^1_n) = n,\tilde{\tau}_n^1<T\land\tau_\infty^\infty\} \text{ or } \tau^\infty_\infty <T\Big)\nonumber\\
    \geq& V(I(0))-V(C_{\beta,\gamma,\delta})-\frac{1}{\delta T}.\label{ineq:explosion prob step 4}
\end{align}
Because $V$ is a strictly increasing function and $I(0)>C_{\beta,\gamma,\delta}$, there exists $T_0>0$ such that for any $T>T_0$ we have 
\begin{align}
    V(I(0))-V(C_{\beta,\gamma,\delta})-\frac{1}{\delta T}>0.\label{ineq:positivity of RHS}
\end{align}
Choose $T>T_0$ then by \eqref{ineq:explosion prob step 4} and \eqref{ineq:positivity of RHS}
\begin{align}
    \begin{split}
        &\mathbb{P}\Big(\{I(\tilde{\tau}^1_n) = n,\tilde{\tau}_n^1<T\land\tau_\infty^\infty\} \text{ or } \tau^\infty_\infty <T\Big)\\
    \geq &V(I(0))-V(C_{\beta,\gamma,\delta})-\frac{1}{\delta T}>0.
    \end{split}\label{ineq:explosion prob}
\end{align}
Because the RHS of \eqref{ineq:explosion prob} doesn't depend on $n$, taking $n$ goes to infinity it follows that either $L^\infty$ norm explodes before time $T$ or $L^1$ norm explodes before time $T$. By triangular inequality
\begin{align}
    I(t)=\int u(t,x) dx\leq C |u|_{L^\infty}.\label{ineq:L1 and Linfty}
\end{align}
This implies that if $L^1$ norm explodes then $L^\infty$ norm must explode. Thus, \eqref{ineq:explosion prob} and \eqref{ineq:L1 and Linfty} implies that
\begin{align}
    \mathbb{P}(\tau_\infty^\infty<\infty)>0.
\end{align}

\section{\texorpdfstring{Proof of Theorem \ref{thm:main result}: Case $\beta\in(1,3),\gamma\in (\frac{\beta}{2},\frac{\beta+3}{4})$}{Proof of Theorem 1: Case gamma in (beta/2,(beta+3)/4)}} \label{S:non-holder-case}
We will prove Theorem \ref{thm:main result} in the region B of Figure \ref{fig:3}.
\begin{figure}[!htbp]
    \centering
    \includegraphics[width=\textwidth]{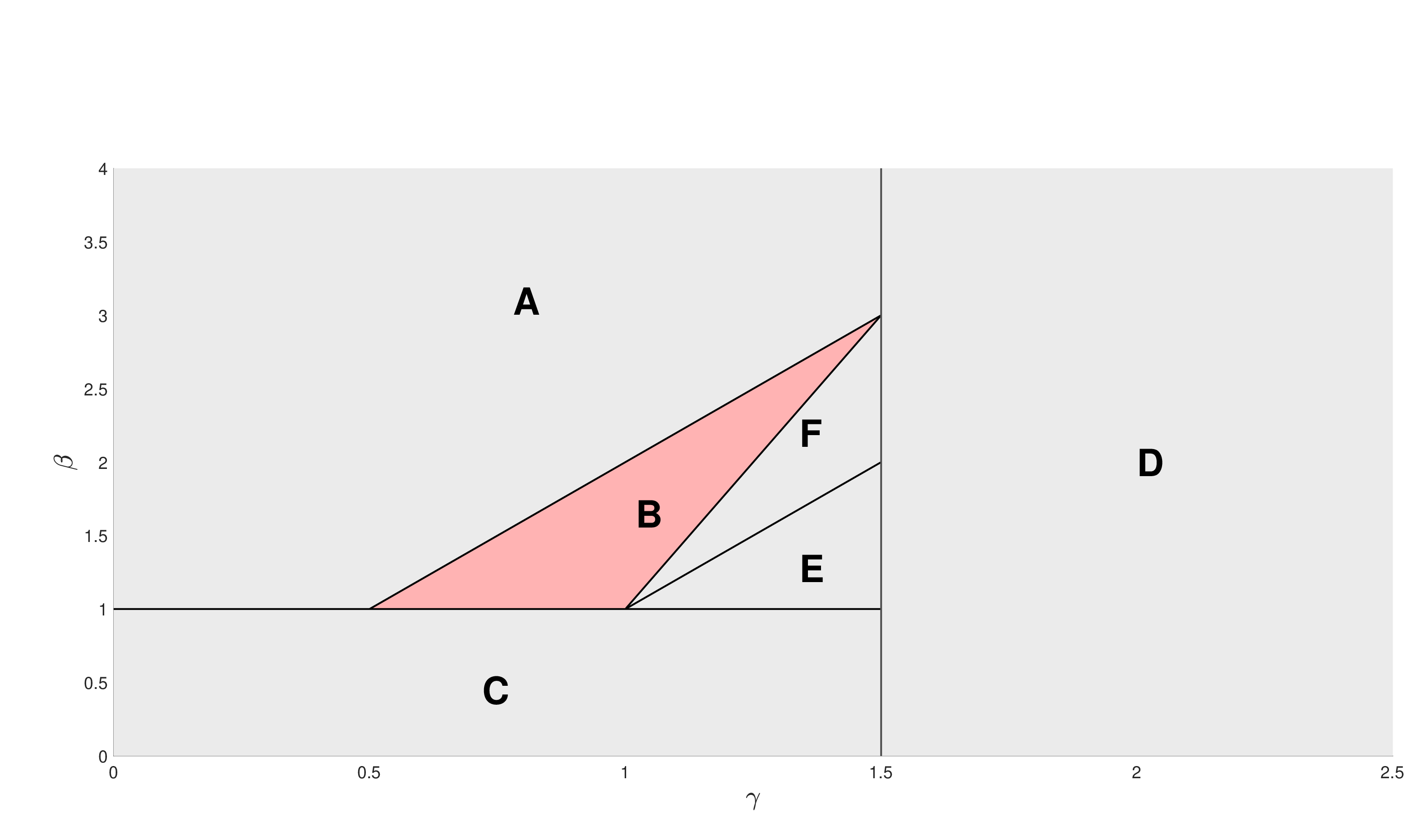}
    \caption{Explosion regions Section \ref{S:non-holder-case}}
    \label{fig:3}
\end{figure}

Recall the local mild solution $u(t,x)$ solves the following integral equation 
\begin{align}
	u(t,x) = &\int G(t,x-y) u(0,y)dy + \int_0^t \int G(t-s,x-y)(u(s,y))^\beta dyds \nonumber\\
	&+ \int_0^t \int G(t-s,x-y)\sigma(u(s,y)) W(dyds). \label{eq:integral equation case 2}
\end{align}
To prove for Theorem \ref{thm:main result}, we need the following estimates applied on the Lebesgue and stochastic convolution integrals in \eqref{eq:integral equation case 2}.
Let $\varphi$ be an adapted random field,
define
\begin{equation}
	Z^{\varphi}(t,x) := \int_0^t\int G(t-s,x-y)\varphi(s,y)W(dyds)\label{def:Z phi}
\end{equation}
and
\begin{equation}
	Y^{\varphi}(t,x) := \int_0^t\int G(t-s,x-y)\varphi(s,y)dyds.\label{def:Y phi}
\end{equation}

\begin{theorem}[Theorem 1.2 of \cite{salins2025}] \label{thm:stoch}
	For any $p>6$ there exists a constant $C_p>0$ such that for any $T \in [0,1]$ and any adapted $\varphi$,
	\begin{equation}
		\mathbb{E} \sup_{t \in [0,T]} \sup_{x \in D} |Z^{\varphi}(t,x)|^p \leq C_p T^{\frac{p}{4} - \frac{3}{2}} \mathbb{E}\int_0^T \int |\varphi(s,y)|^pdyds.
	\end{equation}
\end{theorem}
\textbf{Remark}: The constant $C_p$ degrades as $p \downarrow 6$ in the sense that
	\[\lim_{p \downarrow 6} C_p = \infty.\]

We develop a similar estimate for $Y^{\varphi}(t,x)$.
\begin{theorem} \label{thm:Leb}
	For any $p>\frac{3}{2}$, there exists a constant $K_p>0$ such that for any $T\in [0,1]$ and any adapted $\varphi$,
	\begin{equation}
		\mathbb{E} \sup_{t \in [0,T]} \sup_{x \in D} |Y^{\varphi}(t,x)|^p \leq K_p T^{p-\frac{3}{2}} \mathbb{E}\int_0^T \int |\varphi(s,y)|^pdyds
	\end{equation}
\end{theorem}

\begin{proof}
	The proof is a straightforward application of H\"older's inequality. For fixed $t \in [0,T]$ and $x \in D$,
	\begin{equation}
		|Y^{\varphi}(t,x)|^p \leq \left(\int_0^t\int G^{\frac{p}{p-1}}(t-s,x-y)dyds\right)^{p-1}\int_0^t \int |\varphi(s,y)|^pdyds.
	\end{equation}
	We know that $G(t,x)$ is positive, $\int G(t,x)dx =1$ and $G(t,x)\leq Ct^{-\frac{1}{2}}$ for any $t\in(0,T]$ and $x\in D$. Therefore an upper bound is
	\begin{equation}
		|Y^{\varphi}(t,x)|^p \leq C(\left(\int_0^t (t-s)^{-\frac{1}{2(p-1)}}ds\right)^{p-1}\int_0^t \int |\varphi(s,y)|^pdyds.
	\end{equation}
	Because $p>\frac{3}{2}$ this is integrable and
	\begin{equation}
		|Y^{\varphi}(t,x)|^p \leq K_p t^{p-\frac{3}{2}}\int_0^t \int |\varphi(s,y)|^pdyds.
	\end{equation}
	In fact, we know that $K_p = C\left(\frac{2p-2}{2p-3}\right)^{p-1}$.
	The right-hand side of the expression does not contain $x$ and is increasing in $t$. Therefore
	\begin{equation}
		\mathbb{E} \sup_{t \in [0,T]} \sup_{x \in D} |Y^{\varphi}(t,x)|^p \leq K_p T^{p-\frac{3}{2}} \mathbb{E} \int_0^T \int |\varphi(s,y)|^pdyds.
	\end{equation}
\end{proof}

We will apply these theorems with $p$ very close to $6$ and $\frac{3}{2}$ respectively. Also, for simplicity, we focus on positive and adapted random field $\varphi$ for the following Theorems.
\begin{theorem}\label{thm:Stochastic estimate case 2}
Let $\varphi$ be an adapted and positive random field. Recall from \eqref{def:Z phi} we define
\begin{align}
    Z^{\sigma(\varphi)}(t,x) := \int_0^t\int G(t-s,x-y)\sigma(\varphi(s,y))W(dyds).\label{def:Z phi case 2}
\end{align}
If
\begin{align}
    \theta> \frac{1}{3-2\gamma}, \label{eq:theta-lower-1}
\end{align}
then there exists a $p>6$ (but very close to $6$) such that
	\[-\alpha:= \theta((p-2)\gamma -p) + 2 <0.\]
Furthermore, if 
\begin{align}
    &\mathbb{E} \int_0^1 \int |\sigma(\varphi(s,y))|^2dyds \leq 3^{2(n+1)},\label{eq:quad-var-bound case 2}\\
    &\mathbb{P}\Big(\sup_{t\leq 1}\sup_{x\in D}\varphi(t,x)dx\leq 3^{\theta(n+1)}\Big)=1,\label{eq:upper bound of Linf phi}
\end{align}
	then there exists a universal constant $C_p>0$, independent of $n$ such that
	\begin{equation}
		\mathbb{P}\left(\sup_{t \in [0,1]} \sup_{x \in D} Z^{\sigma(\varphi)}(t,x)> 3^{\theta n}\right) \leq C_p 3^{-\alpha n}\label{eq:P for Zphi}
	\end{equation}
    Because the exponent is negative, these probabilities decrease geometrically.
\end{theorem}
\begin{proof}
    We apply Chebyshev's inequality and Theorem \ref{thm:stoch} with a $p$ close enough to $6$ such that
	\[\theta((p-2)\gamma -p) + 2 <0.\]
    Thus, it follows that 
	\begin{align}
		&\mathbb{P}\left(\sup_{t \in [0,1]} \sup_{x \in D} Z^{\sigma(\varphi)}(t,x)> 3^{\theta n}\right) \nonumber\\
		&\leq 3^{-n\theta p} \mathbb{E} \sup_{t \in [0, 1]}\sup_{x \in D} |Z^{\sigma(\varphi)}(t,x)|^p  \nonumber\\
		&\leq C 3^{-n\theta p}  \mathbb{E} \int_0^{1} \int  |\sigma(\varphi(s,y))|^pdyds \nonumber\\
		&\leq C 3^{-n \theta p} 3^{(p-2)\theta \gamma n} \mathbb{E} \int_0^{1} \int |\sigma(u(s,y))|^2dyds.
 	\end{align}
 	In the previous line, we used the fact that $|\sigma(\varphi(s,y))|^p\leq(\varphi(s,y))^{(p-2)\gamma}|\sigma(u(s,y))|^2 $ and factored out $|\varphi|_{L^\infty}^{(p-2)\gamma} \leq 3^{ (p-2)\gamma (n+1) \theta}$ from the integral because of \eqref{eq:upper bound of Linf phi}. 
    Thus, we can use  estimate \eqref{eq:quad-var-bound case 2} to bound
 	\begin{align}
 		\mathbb{P}\left(\sup_{t \in [0,1]} \sup_{x \in D} Z^{\sigma(\varphi)}(t,x)> 3^{\theta n}\right) \leq C 3^{n(\theta((p-2)\gamma - p)+2)}.
 	\end{align}
 	Remember that the exponent is negative.
\end{proof}

Now we prove a similar result for the Lebesgue integral using similar arguments.
\begin{theorem}\label{thm:Lebesgue estimate case 2}
Let $\varphi$ be an adapted and positive random field. Recall from \eqref{def:Y phi} we define
\begin{align}
    Y^{\varphi^{\beta}}(t,x) := \int_0^t\int G(t-s,x-y)(\varphi(s,y))^{\beta}dyds.\label{def:Y phi case 2}
\end{align}
If
\begin{align}
    \theta>\frac{2}{3-\beta}, \label{eq:theta-lower-2}
\end{align}
then there exists a $p>\frac{3}{2}$ (but close to $\frac{3}{2}$) such that
	\[-\delta:=\theta \left((p-1)\beta-p\right) + 1<0.\]
Furthermore, if $\varphi$ satisfies \eqref{eq:upper bound of Linf phi} and
\begin{align}
    \mathbb{E}\int_0^1\int (\varphi(s,y))^{\beta}dyds\leq 3^{n+1},\label{eq:Lbeta bound}
\end{align}
then there exists a constant $K>0$, independent of $n$ such that
	\begin{equation}
		\mathbb{P}\left(\sup_{t \in [0,1]} \sup_{x \in D} |Y^{\varphi^\beta}(t,x)|>3^{\theta n}\right)
		\leq K 3^{-\delta n},\label{eq:P for Yphi}
	\end{equation}
The exponent is negative, so these probabilities decay geometrically in $n$.
\end{theorem}
\begin{proof}
	We apply Chebyshev's inequality and Theorem \ref{thm:Leb} with a $p$ close enough to $\frac{3}{2}$ such that
    \begin{align}
        \theta \left((p-1)\beta-p\right) + 1<0.
    \end{align}
    Thus, it follows that
	\begin{align}
		&\mathbb{P}\left(\sup_{t \in [0,1]} \sup_{x \in D} |Y^{\varphi^\beta}(t,x)|>3^{\theta n}\right) \nonumber\\
		&\leq 3^{-n\theta p} \mathbb{E} \sup_{t \in [0,1]} \sup_{x \in D} |Y^{\varphi^\beta}(t,x)|^p \nonumber\\
		&\leq K 3^{-n \theta p} \mathbb{E} \int_0^{1} \int (\varphi(s,y))^{p \beta}dyds\nonumber\\
        &\leq K 3^{-n \theta p}3^{(p-1)\theta\beta n}\mathbb{E} \int_0^{1} \int (\varphi(s,y))^{\beta}dyds.
	\end{align}
    In the previous line, we used the fact that $(\varphi(s,y))^{p \beta}\leq (\varphi(s,y))^{(p-1) \beta}(\varphi(s,y))^{\beta}$ and factored out $(\varphi(s,y))^{(p-1) \beta}\leq 3^{(p-1)\beta(n+1)\theta}$ from the integral because of \eqref{eq:upper bound of Linf phi}. Thus, we can use estimate \eqref{eq:Lbeta bound} to bound
    \begin{align}
        \mathbb{P}\left(\sup_{t \in [0,1]} \sup_{x \in D} |Y^{\varphi^\beta}(t,x)|>3^{\theta n}\right)
		\leq K 3^{n(\theta((p-1)\beta-p)+1)}.
    \end{align}
    Remember that the exponent is negative.
\end{proof}

The previous two results, combined with Lemmas \ref{lem:Lebesgue int} and  \ref{lem:quad var}, enable us to show that if $\|u(t)\|_{L^1} \approx 3^n$, then it is very unlikely for $\|u(t)\|_{L^\infty}$ to exceed $3^{n\theta}$ when $\theta$ satisfies certain conditions
\begin{lemma}\label{lem: prob of I triples}
    If $\theta$ satisfies \eqref{eq:theta-lower-1}, \eqref{eq:theta-lower-2} and
    \begin{align}
        1\leq \theta< \frac{1}{2\gamma-\beta}, \label{eq:theta-upper}
    \end{align}
    we can choose $\eta$ such that
   \begin{align}
    & \eta<1, \label{eq:eta-upper}\\
    & \eta>(2\gamma-\beta)\theta. \label{eq:eta-lower}
   \end{align}
   Assume the initial condition of the local mild solution $u(t,x)$ satisfies
	\begin{equation}
	I(0) = 3^n \text{ and } |u(0)|_{L^\infty} \leq 3^{\theta n},\label{eq:initial condition case 2}
\end{equation}
for $n\geq n_0$ where $n_0\in \mathbb{N}_+$ satisfies
\begin{align}
    \frac{\beta}{2}\cdot 3^{(2\gamma-\beta)\theta(n_0+1)-\eta n_0}<1,\label{ineq: assumption for n0}
\end{align}
and define a stopping time $\tau_n$ where
    \begin{align}
    \tau_n:=\inf\{t>0: I(t)> 3^{n+1} \text{ or }I(t)\leq 3^{\eta n}\text{ or } |u(t)|_{L^\infty}> 3^{(n+1)\theta}\},\label{def:tau_n case 2}
    \end{align}
then with $T_n=3^{-nx}$ for any $x\in (0,\beta-1)$ there exists $\psi>0$ and $C>0$ independent of $n$ such that
    \begin{align}
        \mathbb{P}\Big(I(\tau_n)=3^{n+1},\tau_n\leq T_n\Big)\geq 1 -C 3^{-\psi n}.
    \end{align}
\end{lemma}
Furthermore, we can choose $n_0$ in a way such that $1-C3^{-\psi n}\geq 1-C3^{-\psi n_0}>0$ for $n\geq n_0$.

\begin{proof}
    By \eqref{eq:Ito of V} it follows that for any $\epsilon>0$
\begin{align}
     \begin{split}
         V(I(t\land \tau_n))=&V(I(0))+\int_0^{t\land \tau_n}\int \frac{\epsilon (u(s,x))^{\beta}}{(I(s))^{\epsilon+1}}-\frac{\epsilon(\epsilon+1)(\sigma(u(s,x))^2}{2(I(s))^{\epsilon+2}}dxds\\
         &+N(t\land \tau_n).
     \end{split}\label{eq: Ito of I to the power -epsilon}
\end{align}
For simplicity we drop the subscript of $I_n(t)$ in the drift term of \eqref{eq: Ito of I to the power -epsilon}. By $(\sigma(u(s,x))^2\leq (u(s,x))^{2\gamma}\leq (u(s,x))^{\beta}|u(s)|_{L^\infty}^{2\gamma-\beta}$, we have 
\begin{align}
   \begin{split}
        B:=&\int \frac{\epsilon (u(s,x))^{\beta}}{(I(s))^{\epsilon+1}}-\frac{\epsilon(\epsilon+1)(\sigma(u(s,x)))^2}{2(I(s))^{\epsilon+2}}dx\\
    \geq&\int \frac{\epsilon (u(s,x))^{\beta}}{(I(s))^{\epsilon+1}}-\frac{\epsilon(\epsilon+1)(u(s,x))^{\beta}|u(s)|_{L^\infty}^{2\gamma-\beta}}{2(I(s))^{\epsilon+2}}dx.
   \end{split}\label{ineq:drift step 1}
\end{align}
By the definition of $\tau_n$, it follows that $|u(s)|_{L^\infty}^{2\gamma-\beta}\leq 3^{(2\gamma-\beta)(n+1)\theta}$ and \eqref{ineq:drift step 1} becomes
\begin{align}
    B\geq & \epsilon\Big(1-\frac{(\epsilon+1)}{2I(s)}3^{(2\gamma-\beta)(n+1)\theta}\Big)\int \frac{ (u(s,x))^{\beta}}{(I(s))^{\epsilon+1}}dx.\label{ineq:drift step 2}
\end{align}
Also, because $s\leq t\land \tau_n$ we have $I(t)\geq 3^{\eta n}$ by the definition of $\tau_n$. Thus,
\begin{align}
    B\geq &\epsilon\Big(1-\frac{(\epsilon+1)}{2}3^{(2\gamma-\beta)(n+1)\theta-\eta n}\Big)\int \frac{ (u(s,x))^{\beta}}{(I(s))^{\epsilon+1}}dx.\label{ineq: drift step 3}
\end{align}
It follows that by H\"older inequality
\begin{align}
     B\geq & C\Big(1-\frac{(\epsilon+1)}{2}3^{(2\gamma-\beta)(n+1)\theta-\eta n}\Big)(I(s))^{\beta-\epsilon-1}.\label{ineq: drift}
\end{align}
By \eqref{eq:eta-lower}, the coefficient of $n$ on the exponent is negative . Choose $\epsilon=\beta-1$ it follows that by \eqref{ineq: assumption for n0}
\begin{align}
    \int_0^{t\land \tau_n}\int \frac{\epsilon (u(s,x))^{\beta}}{(I(s))^{\epsilon+1}}-\frac{\epsilon(\epsilon+1)(\sigma(u(s,x)))^2}{2(I(s))^{\epsilon+2}}dxds\geq C(t\land\tau_n).
\end{align}
Thus, by \eqref{eq: Ito of I to the power -epsilon} we have
\begin{align}
    C\mathbb{E}(t\land \tau_n)\leq \mathbb{E}(I(0))^{-\epsilon}=3^{-\epsilon n}=3^{-(\beta-1)n}.
\end{align}
Thus, by Chebyshev inequality it follows that for any $t>0$
\begin{align}
    \mathbb{P}(\tau_n>t)\leq \frac{3^{-(\beta-1)n}}{Ct}.\label{ineq: chebyshev for tau}
\end{align}
Also, by submartingale property of \eqref{eq: Ito of I to the power -epsilon} we have
\begin{align}
    V(I(0))\leq \mathbb{E}\Big(V(I(t\wedge\tau))\Big).
\end{align}
Thus, by applying $V(x):=1-x^{-\epsilon}$ it follows that
\begin{align}
      3^{-n\epsilon}=(I(0))^{-\epsilon}&\geq \mathbb{E}(I(t\land \tau))^{-\epsilon}\geq 3^{-n\eta\epsilon}\mathbb{P}(I(\tau)=3^{\eta n},\tau_n<t)
\end{align}
Thus,
\begin{align}
    \mathbb{P}(\inf_{t\in [0,t\land\tau_n]}I(t)\leq 3^{\eta n})\leq 3^{-n(\beta-1)(1-\eta)}.\label{ineq: down probability case 2}
\end{align}
Furthermore, define
\begin{align}
    \begin{split}
         u(t,x)=&\int G(t,x-y)u(0,y)dy+\int_0^t\int G(t-s,x-y)(u(s,y))^{\beta}dyds\\
    &+\int_0^t\int G(t-s,x-y)\sigma(u(s,y))W(dyds)\\
    =:&S(t,x)+Y(t,x)+Z(t,x).
    \end{split}
\end{align}
Let $\varphi(t,x):=u(t,x)\mathbbm{1}_{\{t\leq \tau_n\}}$ and use Lemma \ref{lem:Lebesgue int} and \ref{lem:quad var} to see that $\varphi$ satisfies \eqref{eq:quad-var-bound case 2},\eqref{eq:upper bound of Linf phi} and \eqref{eq:Lbeta bound}. Thus, we can directly apply Theorem \ref{thm:Stochastic estimate case 2} and \ref{thm:Lebesgue estimate case 2}
\begin{align}
    \mathbb{P}(\sup_{t\in [0,1\land \tau_n]}\sup_{x\in D}|Z(t,x)|>3^{\theta n})\leq C_p 3^{-\alpha n}\nonumber\\
    \mathbb{P}(\sup_{t\in [0,1\land \tau_n]}\sup_{x\in D}|Y(t,x)|>3^{\theta n})\leq K3^{-\delta n}.\label{ineq:Y and Z}
\end{align}
Thus, we can replace $[0,1\land\tau_n]$ with $[0,T_n\land\tau_n]$ in the previous inequalities because $T_n=3^{-nx}<1$ where $x\in(0,\beta-1)$. Furthermore, if the event
\begin{align}
    \Big\{\sup_{t\in [0,T_n\land \tau_n]}\sup_{x\in D}|Z(t,x)|< 3^{\theta n}\text{ and }\sup_{t\in [0,T_n\land \tau_n]}\sup_{x\in D}|Y(t,x)|< 3^{\theta n}\Big\}\label{event:Y and Z}
\end{align}
happens, then 
\begin{align}
    \Big\{\sup_{t\in[0,T_n\land\tau_n]}\sup_{x\in D}|u(t,x)|<3^{\theta(n+1)}\Big\}\label{event:sup u}
\end{align}
Combining \eqref{ineq:Y and Z},\eqref{event:Y and Z} and \eqref{event:sup u}, it follows that
\begin{align}
    &\mathbb{P}(\sup_{t\in [0,T_n\land\tau_n]}|u(t)|_{L^\infty}>3^{\theta(n+1)})\nonumber\\
    \leq & \mathbb{P}(\sup_{t\in [0,1\land \tau_n]}\sup_{x\in D}|Z(t,x)|>3^{\theta n})+\mathbb{P}(\sup_{t\in [0,1\land \tau_n]}\sup_{x\in D}|Y(t,x)|>3^{\theta n})\nonumber\\
    \leq & C_p3^{-\alpha n}+K3^{-\delta n}.\label{ineq:up probability case 2} 
\end{align}
Choose $t=T_n$ then from the definition of $\tau_n$, there are four possibilities
\begin{align}
    I(\tau_n)=3^{\eta n}, I(\tau_n)=3^{n+1},|u(\tau_n)|_{L^\infty}=3^{\theta(n+1)}\text{ or }\tau_n>T_n.
\end{align}
Combing \eqref{ineq: chebyshev for tau},\eqref{ineq: down probability case 2} with $t=T_n$ and \eqref{ineq:up probability case 2}, it follows that
\begin{align}
    \begin{split}
        \mathbb{P}(I(\tau_n)\ne 3^{n+1})\leq &\mathbb{P}(\inf_{t\in [0,T_n\land\tau_n]}I(t)\leq 3^{\eta n})\\
        &+\mathbb{P}(\tau_n>T_n)+\mathbb{P}(\sup_{t\in [0,T_n\land\tau_n]}|u(t)|_{L^\infty}>3^{\theta(n+1)})\\
    \leq & C3^{-\phi n}
    \end{split}\label{ineq:Itau n doesn't triple}
\end{align}
for some $\phi>0$. Thus, by \eqref{ineq: chebyshev for tau} with $t=T_n=3^{-nx}$ and \eqref{ineq:Itau n doesn't triple}
\begin{align}
    \begin{split}
        \mathbb{P}(I(\tau_n)= 3^{n+1},\tau_n\leq T_n)&=\mathbb{P}(\tau_n\leq T_n)-\mathbb{P}( I(\tau_n)\ne 3^{n+1},\tau_n\leq T_n)\\
        &\geq 1-\mathbb{P}(\tau_n>T_n)-\mathbb{P}(I(\tau_n)\neq 3^{n+1})\\
        &\geq 1-C3^{-\psi n}
    \end{split}
\end{align}
for some $\psi>0$. Furthermore, combining all of the estimates it follows that we can choose $n_0$ in a way that $1-C3^{-\psi n}\geq 1-C3^{-\psi n_0}>0$ for any $n\geq n_0$.
\end{proof}

To prove Theorem \ref{thm:main result}: Case $\beta\in(1,3),\gamma\in (\frac{\beta}{2},\frac{\beta+3}{4})$ we need to define the following sequence of stopping times. Assume the initial condition of the local mild solution $u(t,x)$ satisfies
\begin{equation}
	I(0) = 3^{n_0} \text{ and } |u(0)|_{L^\infty} \leq 3^{\theta n_0},\label{eq:initial condition n0}
\end{equation}
for some $n_0\in \mathbb{N}_+$ satisfies \eqref{ineq: assumption for n0} and $1-C3^{-\psi n_0}>0$ where $\theta$ satisfies \eqref{eq:theta-lower-1},\eqref{eq:theta-lower-2} and \eqref{eq:theta-upper} and $\eta$ satisfies \eqref{eq:eta-upper} and \eqref{eq:eta-lower}. Define $\tilde{\tau}_{n_0-1}:=0$ and 
\begin{align}
    \tilde{\tau}_n:=
        \inf\{t>\tilde{\tau}_{n-1}: I(t)> 3^{n+1} \text{ or }I(t)\leq 3^{\eta n}\text{ or } |u(t)|_{L^\infty}> 3^{(n+1)\theta}\} \label{def:seq of tau n}
\end{align}
for $n\geq n_0$.
\begin{corollary}\label{cor: prob of I triples}
    Assume the initial condition of the local mild solution $u(t,x)$ satisfies \eqref{eq:initial condition n0}, $\theta$ satisfies \eqref{eq:theta-lower-1},\eqref{eq:theta-lower-2} and \eqref{eq:theta-upper} and $\eta$ satisfies \eqref{eq:eta-upper} and \eqref{eq:eta-lower}, define a sequence of stopping time $\{\tilde{\tau}_n\}_{n=n_0-1}^\infty$ with $\tilde{\tau}_{n_0-1}:=0$ and $\tilde{\tau}_n$ defined by \eqref{def:seq of tau n} then with $T_n=3^{-nx}$ for any $x\in (0,\beta-1)$ there exists $\psi>0$ and $C>0$ independent of $n$ such that
    \begin{align}
        \mathbb{P}\Big(I(\tilde{\tau}_n)=3^{n+1},\tilde{\tau}_n-\tilde{\tau}_{n-1}\leq T_n\Big| I(\tilde{\tau}_{n-1})=3^{n}\Big)\geq 1 -C 3^{-\psi n}.\label{ineq: conditional tripling}
    \end{align}
    It is the same as in Lemma \ref{lem: prob of I triples}, we can choose $n_0$ such that $1-C3^{-\psi n_0}>0$.
\end{corollary}
Corollary \ref{cor: prob of I triples} is a conditional version of Lemma \ref{lem: prob of I triples}. By strong Markov property, we prove explosion by showing
\begin{align}
    \mathbb{P}\Big(\bigcap_{n=n_0}^\infty \big\{ I(\tilde{\tau}_n)=3^{n+1},\tilde{\tau}_n-\tilde{\tau}_{n-1}\leq T_n\big\}\Big)>0.
\end{align}
\noindent\textit{Proof of Theorem \ref{thm:main result}: Case $\beta\in(1,3),\gamma\in (\frac{\beta}{2},\frac{\beta+3}{4})$}. 
Assume the initial condition of the local mild solution $u(t,x)$ satisfies \eqref{eq:initial condition n0} then we can directly apply Corollary \ref{cor: prob of I triples} by choosing the corresponding $\theta,\eta$ and a sequence of stopping time $\{\tilde{\tau}_n\}_{n=n_0-1}^\infty$.

Let $A_n:=\{I(\tilde{\tau}_n)=3^{n+1},\tilde{\tau}_n-\tilde{\tau}_{n-1}\leq T_n\}$ for $n\geq n_0$ and $A_{n_0-1}:=\{I(\tilde{\tau}_{n_0-1})=3^{n_0}\}$ then by strong Markov property and \eqref{ineq: conditional tripling} it follows that
\begin{align}
    \begin{split}
        \mathbb{P}(A_n|A_{n-1})&= \mathbb{P}\Big(I(\tilde{\tau}_n)=3^{n+1},\tilde{\tau}_n-\tilde{\tau}_{n-1}\leq T_n\Big| I(\tilde{\tau}_{n-1})=3^{n}\Big)\\
    &\geq 1 -C 3^{-\psi n}.
    \end{split}\label{eq:conditional prob case 2}
\end{align}
Remember we choose $n_0$ such that $1-C3^{-\psi n_0}>0$. Recall the fact that for a product of real numbers $a_n\in (0,1)$ we have
\begin{align}
    \prod_{n=1}^\infty a_n>0 \text{ if and only if }\sum_{n=1}^\infty (1-a_n)<\infty,\label{eq:real num fact}
\end{align}
it follows that by \eqref{eq:conditional prob case 2} and \eqref{eq:real num fact}
\begin{align}
    \prod_{n=n_0}^\infty \mathbb{P}(A_n|A_{n-1})>0,
\end{align}
Thus, by strong Markov property
\begin{align}
    &\mathbb{P}\Big(\bigcap_{n=n_0}^\infty \{I(\tilde{\tau}_n)=3^{n+1},\tilde{\tau}_n-\tilde{\tau}_{n-1}\leq T_n\}\Big)\nonumber\\
    =&\prod_{n=n_0}^\infty \mathbb{P}(A_n|A_{n-1})>0.
\end{align}
Because $\sum_{n=n_0}^\infty T_n<\infty$, the $L^1$ norm of the solution, $I(t)$, reaches $\infty$ in finite time with postive probability. This implies that
\begin{align}
    \mathbb{P}(\tau_{\infty}^\infty<\infty)>0.
\end{align}

\newpage
\bibliographystyle{abbrv} 
\bibliography{Bioliography}

@article{Franzova1999,
author = {N. Franzova},
title = {Long time existence for the heat equation with a spatially correlated noise term},
JOURNAL = {Stochastic Anal. Appl.},
FJOURNAL = {Stochastic Analysis and Applications},
volume = {17},
number = {2},
pages = {169--190},
year = {1999},
publisher = {Taylor \& Francis},
doi = {10.1080/07362999908809596},


URL = { 
    
        https://doi.org/10.1080/07362999908809596
    
    

},
eprint = { 
    
        https://doi.org/10.1080/07362999908809596
    
    

}

}

@article{Krylov1996,
author = {Krylov, N. V.},
title = {On {$L_p $}-Theory of Stochastic Partial Differential Equations in the Whole Space},
JOURNAL = {SIAM J. Math. Anal.},
FJOURNAL = {SIAM Journal on Mathematical Analysis},
volume = {27},
number = {2},
pages = {313-340},
year = {1996},
doi = {10.1137/S0036141094263317},

URL = { 
    
        https://doi.org/10.1137/S0036141094263317
    
    

},
eprint = { 
    
        https://doi.org/10.1137/S0036141094263317
    
    

}
}

@article {Bezdek2018,
    AUTHOR = {Pavel Bezdek},
     TITLE = {Existence and blow-up of solutions to the fractional
              stochastic heat equations},
   JOURNAL = {Stoch. Partial Differ. Equ. Anal. Comput.},
  FJOURNAL = {Stochastic Partial Differential Equations. Analysis and
              Computations},
    VOLUME = {6},
      YEAR = {2018},
    NUMBER = {1},
     PAGES = {73--108},
      ISSN = {2194-0401,2194-041X},
   MRCLASS = {35R60 (35A01 35B44 35R11 60H15)},
  MRNUMBER = {3768995},
MRREVIEWER = {Le\ Chen},
       DOI = {10.1007/s40072-017-0103-8},
       URL = {https://doi.org/10.1007/s40072-017-0103-8},
}

@article {Foondun2019,
    AUTHOR = {Mohammud Foondun and Wei Liu and Erkan Nane},
     TITLE = {Some non-existence results for a class of stochastic partial
              differential equations},
   JOURNAL = {J. Differ. Equ.},
  FJOURNAL = {Journal of Differential Equations},
    VOLUME = {266},
      YEAR = {2019},
    NUMBER = {5},
     PAGES = {2575--2596},
      ISSN = {0022-0396,1090-2732},
   MRCLASS = {35R60 (60H15)},
  MRNUMBER = {3906260},
       DOI = {10.1016/j.jde.2018.08.039},
       URL = {https://doi.org/10.1016/j.jde.2018.08.039},
}

@article {Debuss2002,
    AUTHOR = {de Bouard, A. and Debussche, A.},
     TITLE = {On the effect of a noise on the solutions of the focusing supercritical nonlinear {S}chrödinger equation},
   JOURNAL = {Probab. Theory Related Fields},
  FJOURNAL = {Probability Theory and Related Fields},
    VOLUME = {123},
      YEAR = {2002},
     PAGES = {76-96},
       DOI = {https://doi.org/10.1007/s004400100183},
}

@article{AGRESTI2023,
title = {Reaction-diffusion equations with transport noise and critical superlinear diffusion: Local well-posedness and positivity},
JOURNAL = {J. Differ. Equ.},
FJOURNAL = {Journal of Differential Equations},
volume = {368},
pages = {247-300},
year = {2023},
issn = {0022-0396},
doi = {https://doi.org/10.1016/j.jde.2023.05.038},
url = {https://www.sciencedirect.com/science/article/pii/S0022039623003765},
author = {Antonio Agresti and Mark Veraar}
}

@article{chen.huang:23:superlinear,
  author        = {Chen, Le and Huang, Jingyu},
  title         = {Superlinear stochastic heat equation on {$\Bbb{R}^d$}},
  journal       = {Proc. Amer. Math. Soc.},
  fjournal      = {Proceedings of the American Mathematical Society},
  volume        = {151},
  year          = {2023},
  number        = {9},
  pages         = {4063--4078},
  issn          = {0002-9939},
  mrclass       = {60H15 (35K57 35R60)},
  mrnumber      = {4607649},
  doi           = {10.1090/proc/16436},
  url           = {https://doi.org/10.1090/proc/16436}
}

@article{Dalang2019,
author = {Robert C. Dalang and Davar Khoshnevisan and Tusheng Zhang},
title = {Global solutions to stochastic reaction–diffusion equations with super-linear drift and multiplicative noise},
volume = {47},
JOURNAL = {Ann. Probab.},
FJOURNAL = {The Annals of Probability},
number = {1},
publisher = {Institute of Mathematical Statistics},
pages = {519 -- 559},
year = {2019},
doi = {10.1214/18-AOP1270},
URL = {https://doi.org/10.1214/18-AOP1270}
}

@article{Foondun2021,
author = {Mohammud Foondun and Eulalia Nualart},
title = {The {O}sgood condition for stochastic partial differential equations},
volume = {27},
journal = {Bernoulli},
number = {1},
publisher = {Bernoulli Society for Mathematical Statistics and Probability},
pages = {295 -- 311},
year = {2021},
doi = {10.3150/20-BEJ1240},
URL = {https://doi.org/10.3150/20-BEJ1240}
}

@article {Liang2022,
    AUTHOR = {Liang, Fei and Zhao, Shuangshuang},
     TITLE = {Global existence and finite time blow-up for a stochastic
              non-local reaction-diffusion equation},
   JOURNAL = {J. Geom. Phys.},
  FJOURNAL = {Journal of Geometry and Physics},
    VOLUME = {178},
      YEAR = {2022},
     PAGES = {Paper No. 104577, 21},
      ISSN = {0393-0440,1879-1662},
   MRCLASS = {60H15 (35B50 35R60)},
  MRNUMBER = {4433058},
       DOI = {10.1016/j.geomphys.2022.104577},
       URL = {https://doi.org/10.1016/j.geomphys.2022.104577},
}

@article{Mueller1991,
Author = {Carl Mueller},
Title = { Long time existence for the heat equation with a noise term.},
JOURNAL = {Probab. Theory Related Fields},
FJOURNAL = {Probability Theory and Related Fields},
Volume = {90},
Year = {1991},
Pages = {505–517},
doi = {https://doi.org/10.1007/BF01192141}
}

@article {Mueller199101,
    AUTHOR = {Mueller, Carl},
     TITLE = {On the support of solutions to the heat equation with noise},
   JOURNAL = {Stoch. Stoch. Rep.},
  FJOURNAL = {Stochastics and Stochastic Reports},
    VOLUME = {37},
      YEAR = {1991},
    NUMBER = {4},
     PAGES = {225--245},
      ISSN = {1045-1129},
   MRCLASS = {60H15 (35R60)},
  MRNUMBER = {1149348},
MRREVIEWER = {A.\ Badrikian},
       DOI = {10.1080/17442509108833738},
       URL = {https://doi.org/10.1080/17442509108833738},
}

@article{Mueller1993,
Author = {Carl Mueller and Richard Sowers},
Title = { Blowup for the heat equation with a noise term.},
JOURNAL = {Probab. Theory Related Fields},
FJOURNAL = {Probability Theory and Related Fields},
Volume = {97},
Year = {1993},
Pages = {287-320},
doi = {https://doi.org/10.1007/BF01195068}
}

@article{Mueller1997,
author = {Carl Mueller},
title = {Long time existence for the wave equation with a noise term},
volume = {25},
JOURNAL = {Ann. Probab.},
FJOURNAL = {The Annals of Probability},
number = {1},
publisher = {Institute of Mathematical Statistics},
pages = {133 -- 151},
year = {1997},
doi = {10.1214/aop/1024404282},
URL = {https://doi.org/10.1214/aop/1024404282}
}

@article{Mueller2000,
author = {Carl Mueller},
title = {The critical parameter for the heat equation with a noise term to blow up in finite time},
volume = {28},
JOURNAL = {Ann. Probab.},
FJOURNAL = {The Annals of Probability},
number = {4},
publisher = {Institute of Mathematical Statistics},
pages = {1735 -- 1746},
year = {2000},
doi = {10.1214/aop/1019160505},
}

@article{salins2025,
author = {Michael Salins},
title = {Solutions to the stochastic heat equation with polynomially growing multiplicative noise do not explode in the critical regime},
volume = {53},
JOURNAL = {Ann. Probab.},
FJOURNAL = {The Annals of Probability},
number = {1},
publisher = {Institute of Mathematical Statistics},
pages = {223 -- 238},
year = {2025},
doi = {10.1214/24-AOP1704},
URL = {https://doi.org/10.1214/24-AOP1704}
}

@article{salins2024,
      Title={Global solutions to the stochastic reaction-diffusion equation with superlinear accretive reaction term and superlinear multiplicative noise term on a bounded spatial domain}, 
      Author={Michael Salins},
      Journal = {Trans. Amer. Math. Soc.},
      FJOURNAL = {Transactions of the American Mathematical Society},
      Volume={375},
      Year={2022},
      Pages={8083-8099},
      doi={https://doi.org/10.1090/tran/8763}
}

@article{Kotelenez1992,
Author = {Kotelenez, Peter},
Title = { Comparison methods for a class of function valued stochastic partial differential equations.},
JOURNAL = {Probab. Theory Related Fields},
FJOURNAL = {Probability Theory and Related Fields},
Volume = {93},
Year = {1992},
Pages = {1–19},
doi = {https://doi.org/10.1007/BF01195385}
}

@article{Mueller1998,
Author = {Carl Mueller},
Title = { Long-time existence for signed solutions of the heat equation with a noise term.},
JOURNAL = {Probab. Theory Related Fields},
FJOURNAL = {Probability Theory and Related Fields},
Volume = {110},
Year = {1998},
Pages = {51-68},
doi = {https://doi.org/10.1007/s004400050144}
}

@book{DaPrato_Zabczyk_2014, 
place={Cambridge}, 
edition={2}, 
series={Encyclopedia of Mathematics and its Applications}, 
title={Stochastic Equations in Infinite Dimensions}, 
publisher={Cambridge University Press}, 
author={Da Prato, Giuseppe and Zabczyk, Jerzy}, 
year={2014}, 
collection={Encyclopedia of Mathematics and its Applications},

}

@article {Dalang1999,
    AUTHOR = {Dalang, Robert C.},
     TITLE = {Extending the martingale measure stochastic integral with
              applications to spatially homogeneous s.p.d.e.'s},
   JOURNAL = {Electron. J. Probab.},
  FJOURNAL = {Electronic Journal of Probability},
    VOLUME = {4},
      YEAR = {1999},
     PAGES = {no. 6, 29},
      ISSN = {1083-6489},
   MRCLASS = {60H05 (35R60 60G15 60G48 60H15)},
  MRNUMBER = {1684157},
MRREVIEWER = {Marta\ Sanz Sol\'e},
       DOI = {10.1214/EJP.v4-43},
       URL = {https://doi.org/10.1214/EJP.v4-43},
}

@article{Yuyang2025,
          AUTHOR = {Michael Salins and Yuyang Zhang},
     TITLE = {Nonexplosion for a large class of superlinear stochastic parabolic equations, in arbitrary spatial dimension},
   JOURNAL = {Stoch. Partial Differ. Equ. Anal. Comput.},
  FJOURNAL = {Stochastics and Partial Differential Equations: Analysis and Computations},
      YEAR = {2025},
       DOI = {10.1007/s40072-025-00400-0},
       URL = {https://doi.org/10.1007/s40072-025-00400-0},
}

@article {Bonder2009,
    AUTHOR = {Fern\'andez Bonder, Julian and Groisman, Pablo},
     TITLE = {Time-space white noise eliminates global solutions in
              reaction-diffusion equations},
   JOURNAL = {Phys. D},
  FJOURNAL = {Physica D: Nonlinear Phenomena},
    VOLUME = {238},
      YEAR = {2009},
    NUMBER = {2},
     PAGES = {209--215},
      ISSN = {0167-2789,1872-8022},
   MRCLASS = {35R60 (35K57 60H15)},
  MRNUMBER = {2516340},
       DOI = {10.1016/j.physd.2008.09.005},
       URL = {https://doi.org/10.1016/j.physd.2008.09.005},
}

@article {Osgood1898,
    AUTHOR = {Osgood, W. F.},
     TITLE = {Beweis der {E}xistenz einer {L}\"osung der
              {D}ifferentialgleichung {$\frac{{dy}}{{dx}} = f\left( {x,y}
              \right)$} ohne {H}inzunahme der {C}auchy-{L}ipschitz'schen
              {B}edingung},
   JOURNAL = {Monatsh. Math. Phys.},
  FJOURNAL = {Monatshefte f\"ur Mathematik und Physik},
    VOLUME = {9},
      YEAR = {1898},
    NUMBER = {1},
     PAGES = {331--345},
      ISSN = {1812-8076},
   MRCLASS = {99-04},
  MRNUMBER = {1546565},
       DOI = {10.1007/BF01707876},
       URL = {https://doi.org/10.1007/BF01707876},
}

@article {Foondun2024,
    AUTHOR = {Foondun, Mohammud and Khoshnevisan, Davar and Nualart,
              Eulalia},
     TITLE = {Instantaneous everywhere-blowup of parabolic {SPDE}s},
   JOURNAL = {Probab. Theory Related Fields},
  FJOURNAL = {Probability Theory and Related Fields},
    VOLUME = {190},
      YEAR = {2024},
    NUMBER = {1-2},
     PAGES = {601--624},
      ISSN = {0178-8051,1432-2064},
   MRCLASS = {60H15 (60F05 60H07)},
  MRNUMBER = {4797376},
MRREVIEWER = {Leila\ Setayeshgar},
       DOI = {10.1007/s00440-024-01263-7},
       URL = {https://doi.org/10.1007/s00440-024-01263-7},
}

@article {Annie2021,
    AUTHOR = {Millet, Annie and Sanz-Sol\'e, Marta},
     TITLE = {Global solutions to stochastic wave equations with superlinear
              coefficients},
   JOURNAL = {Stochastic Process. Appl.},
  FJOURNAL = {Stochastic Processes and their Applications},
    VOLUME = {139},
      YEAR = {2021},
     PAGES = {175--211},
      ISSN = {0304-4149,1879-209X},
   MRCLASS = {60H15 (35R60 60G17 60G60)},
  MRNUMBER = {4264843},
MRREVIEWER = {Omar\ Mellah},
       DOI = {10.1016/j.spa.2021.05.002},
       URL = {https://doi.org/10.1016/j.spa.2021.05.002},
}

@article {Mickey202409,
    AUTHOR = {Michael Salins},
     TITLE = {Global solutions to the stochastic heat equation with
              superlinear accretive reaction term and polynomially growing
              multiplicative white noise coefficient},
   JOURNAL = {Stoch. Partial Differ. Equ. Anal. Comput.},
  FJOURNAL = {Stochastics and Partial Differential Equations. Analysis and
              Computations},
    VOLUME = {14},
      YEAR = {2026},
    NUMBER = {1},
     PAGES = {109--132},
      ISSN = {2194-0401,2194-041X},
   MRCLASS = {60H15},
  MRNUMBER = {5015569},
       DOI = {10.1007/s40072-025-00351-6},
       URL = {https://doi.org/10.1007/s40072-025-00351-6},
}

@misc{joseph2026,
      title={Blowup for the multiplicative stochastic heat equation with superlinear drift}, 
      author={Mathew Joseph and Shubham Ovhal},
      year={2026},
      eprint={2511.23403},
      archivePrefix={arXiv},
      primaryClass={math.PR},
      url={https://arxiv.org/abs/2511.23403}, 
}

@book {Khasminskii,
    AUTHOR = {Khasminskii, Rafail},
     TITLE = {Stochastic stability of differential equations},
    SERIES = {Stochastic Modelling and Applied Probability},
    VOLUME = {66},
   EDITION = {second},
      NOTE = {With contributions by G. N. Milstein and M. B. Nevelson},
 PUBLISHER = {Springer, Heidelberg},
      YEAR = {2012},
     PAGES = {xviii+339},
      ISBN = {978-3-642-23279-4},
   MRCLASS = {60H10 (34F05 60J60 93E15 94A12)},
  MRNUMBER = {2894052},
MRREVIEWER = {Vivek\ S.\ Borkar},
       DOI = {10.1007/978-3-642-23280-0},
       URL = {https://doi.org/10.1007/978-3-642-23280-0},
}

@article {Kaplan,
    AUTHOR = {Kaplan, Stanley},
     TITLE = {On the growth of solutions of quasi-linear parabolic
              equations},
   JOURNAL = {Comm. Pure Appl. Math.},
  FJOURNAL = {Communications on Pure and Applied Mathematics},
    VOLUME = {16},
      YEAR = {1963},
     PAGES = {305--330},
      ISSN = {0010-3640,1097-0312},
   MRCLASS = {35.62},
MRREVIEWER = {P.\ Hartman},
       DOI = {10.1002/cpa.3160160307},
       URL = {https://doi.org/10.1002/cpa.3160160307},
}

@article {Bally1995,
    AUTHOR = {Bally, Vlad and Millet, Annie and Sanz-Sol\'e, Marta},
     TITLE = {Approximation and support theorem in {H}\"older norm for
              parabolic stochastic partial differential equations},
   JOURNAL = {Ann. Probab.},
  FJOURNAL = {The Annals of Probability},
    VOLUME = {23},
      YEAR = {1995},
    NUMBER = {1},
     PAGES = {178--222},
      ISSN = {0091-1798,2168-894X},
   MRCLASS = {60H15 (35K10 35R60)},
  MRNUMBER = {1330767},
MRREVIEWER = {Bj\"orn\ Schmalfuss},
       URL =
              {http://links.jstor.org/sici?sici=0091-1798(199501)23:1<178:AASTIH>2.0.CO;2-N&origin=MSN},
}
\end{document}